\documentclass[a4paper,12pt]{amsart}

\input xypic
\usepackage{amssymb}
\xyoption{all}

\newtheorem{thm}{Theorem}[section]

\newtheorem{cor}[thm]{Corollary}
\newtheorem{prop}[thm]{Proposition}
\theoremstyle{definition}

\def\BZ{{\mathbf{Z}}}

\def\Hom{\operatorname{Hom}\nolimits}

\def\mMod{\operatorname{\!-mod}\nolimits}
\def\MMod{\operatorname{\!-Mod}\nolimits}

\author{Salah Al-Nofayee and Jeremy Rickard} \date\today
\title{Rigidity of tilting complexes and derived equivalence for
  self-injective algebras}
\begin{document}
\maketitle

For any ring $\Lambda$, we shall denote by $\Lambda\MMod$ the category
of all (left) $\Lambda$-modules, and by $\Lambda\mMod$ the full
subcategory of finitely presented modules.

Let $\Lambda$ and $\Gamma$ be derived equivalent rings.  That is, the
unbounded derived categories $D(\Lambda\MMod)$ and $D(\Gamma\MMod)$
are equivalent as triangulated categories.  Then many properties of
$\Lambda$ and $\Gamma$ are shared. One example of this, first proved
in \cite[Corollary 5.3]{Ri2}, is that if $\Lambda$ is a symmetric
finite-dimensional algebra over a field $k$, then so is $\Gamma$.  In
fact, the property of being symmetric can be easily characterized in
terms of the derived category. Recall that an object of the derived
category $D(\Lambda\MMod)$ is called perfect if it is isomorphic to a
bounded complex of finitely-generated projective $\Lambda$-modules,
and that the perfect objects $P$ can be characterized as those that
are compact, meaning that the functor $\Hom_{D(\Lambda\MMod)}(P,-)$
preserves arbitrary direct sums, and that if $\Lambda$ is a
finite-dimensional $k$-algebra, then the objects $X$ of
$D(\Lambda\MMod)$ isomorphic to objects of $D(\Lambda\mMod)$ can be
characterized as those objects such that, for every perfect object
$P$, $\Hom_{D(\Lambda\MMod)}(P,X)$ is finite-dimensional. Thus such
objects $P$ and $X$ are both intrinsically characterized in terms of
the triangulated category $D(\Lambda\MMod)$.

\begin{prop}\label{symm}
  A finite-dimensional $k$-algebra $\Lambda$ is symmetric if and only
  if, for every object $X$ of $D(\Lambda\mMod)$ and every perfect
  object $P$ of $D(\Lambda\MMod)$, $\Hom_{D(\Lambda\MMod)}(P,M)$ and
  $\Hom_{D(\Lambda\MMod)}(M,P)$ are naturally dual as $k$-vector
  spaces.
\end{prop}

\begin{proof}
  It was shown in \cite[Corollary 3.2]{Ri3} that if $\Lambda$ is
  symmetric, then this duality holds.  Conversely, if there is such a
  duality, then the case $P=X=\Lambda$ shows that there is an
  isomorphism between $\Lambda$ and its dual, which, by naturality, is
  an isomorphism of bimodules. Thus $\Lambda$ is symmetric.
\end{proof}

For the weaker property of being self-injective, there seems to be no
such simple direct characterization in terms of the derived
category. However, our main theorem in this paper is that, at least
for finite-dimensional algebras over an algebraically closed field,
the property of being self-injective is preserved by derived
equivalence. The proof depends on a theorem of Huisgen-Zimmermann
and Saorin\cite{HS} on rigidity of tilting complexes. In Section 1 we
show how this theorem also follows from the results of \cite{Ri1}.

\section{Rigidity for tilting complexes}

Let $k$ be an algebraically-closed field and let $\Lambda$ be a
finite-dimensional $k$-algebra.  In Corollary 9 of \cite{HS},
Huisgen-Zimmermann and Saor\'in prove a result that easily implies the
following theorem.

\begin{thm}\label{finite}
  Let $\dots,P^{-1},P^0,P^1,\dots$ be a sequence of finitely-generated
  projective $\Lambda$-modules, such that $P^i=0$ for all but finitely
  many $i$. Then up to isomorphism there are only finitely many
  tilting complexes
$$P = \dots\to P^{-1}\to P^0\to P^1\to\dots.$$
\end{thm}

Since, up to isomorphism, there is only a countable set of
possibilities for the sequence $\dots,P^{-1},P^0,P^1,\dots$, and since
every algebra derived equivalent to $\Lambda$ is isomorphic to the
endomorphism algebra of a tilting complex, the following corollary
follows.

\begin{cor}
  Let $\Lambda$ be a finite-dimensional algebra over an algebraically
  closed field $k$.  Then up to isomorphism there are only countably
  many tilting complexes for $\Lambda$, and hence only countably many
  algebras derived equivalent to $\Lambda$.
\end{cor}

In particular, this rules out (at least over an uncountable field) the
possibility of a family of non-isomorphic but derived equivalent
algebras, parametrized by an algebraic curve or variety of higher
dimension.  A number of people have noticed that Theorem~\ref{finite}
also follows using the techniques of \cite{Ri1}, but we do not know of
any reference to this in the literature, so let us sketch an
alternative proof along these lines.  The complexes of the form
$$P = \dots\to P^{-1}\stackrel{d^{-1}}{\to}P^0\stackrel{d^{0}}{\to}P^1\to\dots$$ 
are parametrized by the $k$-rational points of an affine variety $V$:
namely, the closed sub-variety of
$$\prod_{i\in\BZ}\Hom_{\Lambda}\left(P^i,P^{i+1}\right)$$
determined by the equations $d^{i+1}d^i=0$.

Let $k[V]$ be the coordinate ring of this variety. Then there is a
natural complex
$$\tilde{P}=\dots\to P^{-1}\otimes_k k[V ]\to P^0\otimes_k k[V]\to P^1\otimes_k k[V]\to\dots$$ 
of $\Lambda\otimes_k k[V]$-modules such that, if $P_v$ is the complex
corresponding to a $k$-rational point $v\in V$ with associated maximal
ideal $\mathfrak{m}_v$, then
$$P_v\cong\tilde{P}\otimes_{k[V]}(k[V]/\mathfrak{m}_v).$$

For each $k$-rational point $w\in V$, there is also a complex
$$\tilde{P}_w= \tilde{P}\otimes_k k[V]$$
with the same terms as $\tilde{P}$, but with
$$\tilde{P}_w\otimes_{k[V]}(k[V]/\mathfrak{m}_v)\cong P_w$$
for every $k$-rational point $v\in V$.

Now suppose that $w$ is a $k$-rational point for which $P_w$ is a
tilting complex, and let $\widehat{k[V]_w}$ be the completion of
$k[V]$ at the maximal ideal $\mathfrak{m}_w$.  Then by \cite[Theorem
3.3]{Ri1}, the complexes $\tilde{P}\otimes_{k[V]}\widehat{k[V]_w}$ and
$\tilde{P}_w\otimes_{k[V]}\widehat{k[V]_w}$ are isomorphic.  Using the
Artin Approximation Theorem as in the proof of \cite[Theorem
4.1]{Ri1}, it follows that for some \'etale neighbourhood $U$ of $w$,
there is an isomorphism between $\tilde{P}\otimes_{k[V]}k[U]$ and
$\tilde{P}_w\otimes_{k[V]}k[U]$.  In particular, for any $k$-rational
point $v$ in the open set that is the image of $U\to V$,
$$P_v = \tilde{P}\otimes_{k[V]}(k[V]/\mathfrak{m}_v)
\cong\tilde{P}_w\otimes_{k[V]}(k[V]/\mathfrak{m}_v)
\cong P_w.$$

Thus, the complexes parametrized by the $k$-rational points of some open
neighbourhood of $w$ are all isomorphic to $P_w$, and it follows that
there can only be finitely many isomorphism classes of complexes $P_v$
that are tilting complexes.  In fact, as in Huisgen-Zimmermann and
Saor\'in's proof, it is only the property
$$\Hom_{D(\Lambda\MMod)}(P,P[1])=0$$
that is needed, so in fact there are only finitely many isomorphism
classes of complexes $P_v$ satisfying this weaker condition.

\section{The main theorem}
For a finite-dimensional $k$-algebra $\Lambda$ we define the Nakayama functor
$$\nu_{\Lambda}: \Lambda\MMod \to \Lambda\MMod$$
by
$$\nu_{\Lambda}(M) =\Lambda^{\vee}\otimes_{\Lambda}M,$$
where $\Lambda^{\vee}$ denotes the $k$-linear dual of $\Lambda$,
regarded as a $\Lambda$-bimodule. Thus if $\Lambda$ is symmetric then
the Nakayama functor is isomorphic to the identity functor.  In
general, $\nu_{\Lambda}$ induces an equivalence between the categories
of finitely generated projective and injective modules for
$\Lambda$. It is a right exact functor, and has a total left
derived functor 
$${\bf L}\nu_{\Lambda}: D(\Lambda\MMod)\to D(\Lambda\MMod),$$
where ${\bf L}\nu_{\Lambda}(X)$ is constructed in the usual way by applying
the functor $-\otimes_{\Lambda}X$ to a projective resolution of
$\Lambda^{\vee}$, or by applying the functor
$\Lambda^{\vee}\otimes_{\Lambda}-$ to a projective resolution of $X$.
If $\Lambda$ is self-injective, however, $\nu_{\Lambda}$ is exact, and
so $\nu_{\Lambda}= {\bf L}\nu_{\Lambda}$. In this case $\nu_{\Lambda}$ is a
self-equivalence of the module category $\Lambda\MMod$ and therefore
induces a self-equivalence of the derived category $D(\Lambda\MMod)$.
In \cite[Proposition 5.2]{Ri2}, it was shown that, if $\Lambda$ and
$\Gamma$ are finite-dimensional $k$-algebras and
$$F:D^-(\Lambda\MMod)\to D^-(\Gamma\MMod)$$
is an equivalence of triangulated categories, then, at least if we
replace $F$ by another ``standard'' equivalence that agrees with $F$
on each object up to a (not-necessarily natural) isomorphism, the diagram
$$\xymatrix{
D^-(\Lambda\MMod)\ar[r]^{{\bf L}\nu_{\Lambda}}\ar[d]
& D^-(\Lambda\MMod)\ar[d]\\
D^-(\Gamma\MMod)\ar[r]^{{\bf L}\nu_{\Gamma}}
& D^-(\Gamma\MMod)
}$$
commutes up to isomorphism of functors, and so
$$F({\bf L}\nu_{\Lambda}X)\cong {\bf L}\nu_{\Gamma}(FX)$$
for any object $X$ of $D^-(\Lambda\MMod)$.

\begin{thm}\label{main}
  Let $\Lambda$ and $\Gamma$ be derived equivalent finite-dimensional
  algebras over an algebraically closed field $k$, with $\Lambda$
  self-injective. Then $\Gamma$ is self-injective.
\end{thm}

\begin{proof}
  Since $\nu_{\Gamma}$ induces an equivalence between the categories
  of finitely generated projective and finitely generated injective
  $\Gamma$-modules, we just need to show that $\nu_{\Gamma}(\Gamma)$
  is projective.  

  Since $\nu_{\Lambda}$ is a self-equivalence of the module category
  $\Lambda\MMod$, some power $\nu_{\Lambda}^r$ of the Nakayama functor
  has the property that $\nu_{\Lambda}^r(P)\cong P$ for every
  projective $\Lambda$-module $P$.

  Let $F:D^-(\Lambda\MMod)\to D^-(\Gamma\MMod)$ be a derived
  equivalence.  Then $F^{-1}(\Gamma)$ is isomorphic to some tilting
  complex $T$ for $\Lambda$, and $\nu_{\Lambda}^r(T)$ is also a
  tilting complex which has components isomorphic to those of
  $T$. Thus, for some multiple $s$ of $r$, $\nu_{\Lambda}^s(T)$ is
  isomorphic to $T$, by Theorem~\ref{finite}. Replacing $s$ by some
  further multiple if necessary, we can assume that
  $\nu_{\Lambda}^s(T_i)$ is isomorphic to $T_i$ for every
  indecomposable direct summand $T_i$ of $T$, since there are only
  finitely many such summands up to isomorphism.

  By the discussion preceding the statement of Theorem~\ref{main}, it
  follows that ${\bf L}\nu_{\Gamma}$ is a self-equivalence of
  $D^-(\Gamma\MMod)$ and that ${\bf L}\nu_{\Gamma}^s(Q)$ is isomorphic
  to $Q$ for every indecomposable projective $\Gamma$-module $Q$. We
  shall prove that ${\bf L}\nu_{\Gamma}^n(Q)$ is isomorphic to a
  projective module concentrated in degree zero for all $0\leq n\leq
  s$.  For if not, then let $n<s$ be maximal such that this is not
  true.  

  Since ${\bf L}\nu_{\Gamma}$ preserves the property of having zero
  homology in positive degrees, ${\bf L}\nu_{\Gamma}^n(Q)$ has zero
  homology in positive degrees. Let
$$\dots\to0\to Q^{-k}\to Q^{-k+1}\to\dots\to Q^0\to 0\to\dots$$
be a minimal projective resolution of ${\bf L}\nu_{\Gamma}^n(Q)$ with
$k>0$ and $Q^{-k}\not\cong0$.  It is an indecomposable bounded complex
of finitely generated projective modules, since ${\bf L}\nu_{\Gamma}$
is a self-equivalence of $D^-(\Gamma\MMod)$ and therefore preserves
the property of being an indecomposable perfect object.

Since $\nu_{\Gamma}$ is an equivalence between the categories of
finitely generated projective and injective $\Gamma$-modules, when we
apply $\nu_{\Gamma}$ we get an indecomposable complex
$$\dots\to 0\to\nu_{\Gamma}Q^{-k}\to\nu_{\Gamma}Q^{-k+1}\to\dots
\to\nu_{\Gamma}Q^0\to0\to\dots$$ of injective $\Gamma$-modules with
$\nu_{\Gamma}Q^{-k}\not\cong0$. But this complex is isomorphic in
$D(\Gamma\MMod)$ to ${\bf L}\nu_{\Gamma}^{n+1}(Q$), which, by the
assumption on $n$, has homology concentrated in degree zero.  Thus
the differential
$$\nu_{\Gamma}Q^{-k}\to\nu_{\Gamma}Q^{-k+1}$$
is injective, and therefore split, since $\nu_{\Gamma}Q^{-k}$ is an
injective module. But this contradicts the indecomposability of the
complex.  By contradiction it follows that $\nu_{\Gamma}Q$ is a
projective module for every indecomposable projective $\Gamma$-module
$Q$, and hence $\Gamma$ is a self-injective algebra.
\end{proof}

It would be interesting to find a simple direct characterization of
self-injective algebras in terms of properties of their derived
categories as we did for symmetric algebras in Proposition~\ref{symm},
but it is not clear that this can be done.


\begin{thebibliography}{[Nee]}
\bibitem[Ri1]{Ri1} J.~Rickard,
        {\em Lifting theorems for tilting complexes},
	J. Alg. {\bf 142} (1991), 383--393. 
\bibitem[Ri2]{Ri2} J.~Rickard,
        {\em Derived equivalences as derived functors}, 
        J. London Math. Soc. (2), {\bf 43} (1991), 37--48.
\bibitem[Ri3]{Ri3} J.~Rickard,
        {\em Equivalences of derived categories for symmetric algebras},
        J. Alg. {\bf 257} (2002), 460--481.
      \bibitem[HS]{HS} B.~Huisgen-Zimmermann and M.~Saorin, {\em
          Geometry of chain complexes and outer automorphisms under
          derived equivalence}, Trans. Amer. Math. Soc. {\bf 353}
        (2001), 4757--4777.
\end{thebibliography}
\end{document}